\documentclass{article}

\usepackage{amsmath,amssymb,amsthm,amsfonts,graphics,graphicx,listings,float}
\usepackage[utf8]{inputenc}
\usepackage[usenames,dvipsnames]{color}
\usepackage{tikz}
\usepackage{fullpage}

\newtheorem{theorem}{Theorem}[section]
\newtheorem{lemma}[theorem]{Lemma}

\newtheorem{conjecture}[theorem]{Conjecture}
\newtheorem{proposition}[theorem]{Proposition}
\newtheorem{claim}[theorem]{Claim}

\newtheorem{corollary}[theorem]{Corollary}

\newtheorem{problem}[theorem]{Problem}
\newtheorem{property}[theorem]{Property}

\newtheorem*{problem41}{Problem~\ref{problem too many 2s}}%for main theorem in intro

\begin{document}

\title{Low Diameter Monochromatic Covers of Complete Multipartite Graphs}

\author{Sean English\footnote{Department of Mathematics, University of Illinois at Urbana-Champaign. E-mail: \texttt{senglish@illinois.edu}.} \and Connor Mattes\footnote{Department of Mathematical and Statistical Sciences, University of Colorado Denver. E-mail: \texttt{connor.mattes@ucdenver.edu}.} \and Grace McCourt\footnote{Department of Mathematics, University of Illinois at Urbana-Champaign. E-mail: \texttt{mccourt4@illinois.edu}, research supported in part by NSF RTG Grant DMS-1937241.} \and Michael Phillips\footnote{Department of Mathematical and Statistical Sciences, University of Colorado Denver. E-mail: \texttt{michael.2.phillips@ucdenver.edu}.}}

\date{}

\maketitle

%%%%%%%%%%%%%%%%%%%%%%%%%%%%%%
% Main Content
%%%%%%%%%%%%%%%%%%%%%%%%%%%%%%

% reset page numbering
\setcounter{page}{0}
\pagenumbering{arabic}

\begin{abstract}
	Let the diameter cover number, $D^t_r(G)$, denote the least integer $d$ such that under any $r$-coloring of the edges of the graph $G$, there exists a collection of $t$ monochromatic subgraphs of diameter at most $d$ such that every vertex of $G$ is contained in at least one of the subgraphs. We explore the diameter cover number with two colors and two subgraphs when $G$ is a complete multipartite graph with at least three parts. We determine exactly the value of $D_2^2(G)$ for all complete tripartite graphs $G$, and almost all complete multipartite graphs with more than three parts. 
\end{abstract}

\section{Introduction}\label{sec1}

Given a hypergraph $H$, we say $H$ is $r$-partite if there exists a partition $V(H)=V_1\cup V_2\cup\dots\cup V_r$ such that no edge of $H$ contains two or more vertices from $V_i$ for any $1\leq i\leq r$. The \emph{vertex cover number}, $\tau(H)$, is the minimum cardinality of a set $S\subseteq V(H)$ such that $S$ intersects every edge of $H$, and the \emph{matching number}, $\nu(H)$, is the size of a largest set of pairwise disjoint edges in $H$. Ryser's Conjecture \cite{Ryser} attempts to relate these two parameters in a strong way for $r$-partite graphs.

\begin{conjecture}\label{conjecture Rysers first formulation}
	Let $r\geq 2$ and let $H$ be an $r$-partite hypergraph. Then
	\[
	\tau(H)\leq (r-1)\nu(H).
	\]
\end{conjecture}

Note that the classical K\"onig's Theorem \cite{Konig}, which states that the size of a minimum vertex cover in a bipartite graph $G$ is no bigger than the size of a maximum matching in $G$, confirms the case $r=2$ of Ryser's Conjecture. It was proved by Aharoni \cite{Aharoni} that Ryser's Conjecture also holds for $r=3$. However, already for $r = 4,5$ the only exact result is for the case $\nu(H) = 1$, shown by Tuza \cite{Tuza}, and the strongest result known in general is that there exists some $\epsilon > 0$ such that $\tau(H) \le (r - \epsilon) \nu(H)$, which was proved by Haxell and Scott \cite{45}.

An equivalent formulation of Ryser's Conjecture involves coverings of edge-colored graphs with monochromatic connected components. More specifically, a \emph{monochromatic connected subgraph cover} of an edge colored graph $G$ is a collection of monochromatic connected subgraphs of $G$ whose union contains every vertex. The \emph{monochromatic tree cover number} $\mathrm{tc}_r(G)$ is the least integer such that for any $r$-coloring of the edges of $G$, there exists a monochromatic connected subgraph cover with $\mathrm{tc}_r(G)$ subgraphs. The \emph{independence number} of $G$, $\alpha(G)$, is the size of the largest collection $S\subseteq V(G)$ such that no two vertices in $S$ are adjacent. The following is an equivalent formulation of Conjecture~\ref{conjecture Rysers first formulation}, first observed by Gy\'{a}rf\'{a}s \cite{Gyarfas}.
\begin{conjecture}\label{conjecture Rysers second formulation}
	For every graph $G$ and all $r\geq 2$,
	\[
	\mathrm{tc}_r(G)\leq (r-1)\alpha(G)
	\]
\end{conjecture}
For a brief sketch of the proof that these two conjectures are equivalent, see Section~\ref{section equivalent formulation}.

Phrasing Ryser's Conjecture in terms of edge colorings leads to many interesting questions and connections. In particular, this phrasing of Ryser's Conjecture sheds light on the fact that the conjecture fits into the large body of work on Ramsey theory. In particular, Ryser's Conjecture is closely related to problems involving finding large monochromatic components in edge-colored graphs. Given an $r$-coloring of the edges of the complete graph $K_n$, it is known that there always exists a monochromatic component of size at least $n/(r-1)$~\cite{Gyarfas}, and that this is sharp for infinitely many values of $n$ and $r$. Conjecture~\ref{conjecture Rysers second formulation}, if true, implies this as well, and thus can be thought of as a strengthening of the question of the size of the largest monochromatic component. Large monochromatic components has been studied intensively in both graphs and hypergraphs~\cite{Furedi}, \cite{Haxell}

In this paper, we will explore structural properties of the monochromatic connected subgraphs in a cover, in particular we consider subgraphs with bounded diameter. Assume $t \geq \mathrm{tc}_r(G)$. Then we define the \emph{diameter cover number}, $D_r^t(G)$, to be the least integer $d$ such that in every $r$-coloring of the edges of $G$, there exists a monochromatic connected subgraph cover of $G$ with $t$ subgraphs such that every subgraph has diameter at most $d$.

The diameter cover number was first studied by Mili\'cevi\'c \cite{Milicevic}, who considered 3- and 4-colored complete graphs as well as 2-colored complete bipartite graphs, using the result for 3-colored complete graphs to prove a generalization of Banach's fixed point theorem.

\begin{theorem}[{\cite{Milicevic}}] \
	\begin{enumerate}
		\item For all complete graphs $K_n$, $D_3^2(K_n) \leq 8$.
		\item For all complete graphs $K_n$, $D_4^3(K_n) \leq 80$.
		\item For all complete bipartite graphs $K_{m,n}$, $D_2^2(K_{m,n}) \leq 9$.
	\end{enumerate}
\end{theorem}

The above results were improved and extended by DeBiasio, Kamel, McCourt, and Sheats~\cite{Grace}. 

\begin{theorem}[{\cite{Grace}}] \
	\begin{enumerate}
		\item For all complete graphs $K_n$ with $n\geq 7$, $3 \leq D_3^2(K_n) \leq 4$.
		\item For all complete graphs $K_n$ with $n \geq 5$, $2 \leq D_4^3(K_n) \leq 6$.
		\item For all complete bipartite graphs $K_{m,n}$ with $m \geq 3$, $n\geq 4$, $3 \leq D_2^2(K_{m,n}) \leq 4$.
		\item For all complete bipartite graphs $K_{m,n}$, $D_3^4(K_{m,n}) \leq 6$.
		\item For all graphs $G$ on at least 7 vertices with $\alpha(G) = 2$, $3 \leq D_2^2(G) \leq 6$.
	\end{enumerate}
\end{theorem}

We focus on $2$-colorings of complete multipartite graphs with at least three parts, and prove the following theorem in Section~\ref{sec2}.

\begin{theorem}\label{theorem tripartite upper bound}
	For all $a,b,c\geq 1$,
	\[
	D_2^2(K_{a,b,c})\leq 3.
	\]
\end{theorem}

Note that $tc_2(G) \le 2$ for any complete multipartite graph $G$. This was shown by Chen, Fujita, Gy\'arf\'as, Lehel, and T\'oth \cite{biclique}, and we offer a sketch of a proof of it in section \ref{tc}.

It is worth noting that if $F$ is a spanning subgraph of $G$ and $\mathrm{tc}_r(F)\geq t$, then $D_r^t(F)\geq D_r^t(G)$, so the bound in Theorem~\ref{theorem tripartite upper bound} also applies to all complete $k$-partite graphs with $k \geq 3$.

In Section~\ref{sec3},  we improve on Theorem~\ref{theorem tripartite upper bound}, and provide a complete characterization of the value of $D_2^2(G)$ for all complete tripartite graphs $G$. In addition, for every $k\geq 4$, we determine the value of $D_2^2(G)$ for all but finitely many complete $k$-partite graphs $G$. We also explore the problem of categorizing more values of $D_2^2(G)$. In particular, it is easy to show that if $G$ has a vertex adjacent to all other vertices then $D_2^2(G)\leq 2$. However, it is unclear if  $D_2^2(G)\leq 2$ for any complete $k$-partite graphs without a part of size $1$, when $k$ is large. Let $G_k$ denote the complete multipartite graph with $k$ parts each of size $2$. We propose the following conjecture:
\begin{conjecture}\label{conjecture too many 2s}
	For all $k \ge 3$,
	\[ D_2^2(G_k) = 2.  \]
\end{conjecture}

In Section~\ref{sec4} we provide partial results towards proving this conjecture.

\subsection{Definitions and Notations}

Given a graph $G$ and a vertex $v\in V(G)$, let $N(v)$ denote the open neighborhood of $v$. Given a graph $G$ with $V(G)=\{v_1,v_2,\dots,v_n\}$, a \emph{blow-up} of $G$ is a graph $F$ with a partition $V(F)=V_1\cup V_2\cup\dots\cup V_n$ of non-empty vertex sets such that for all $a\in V_i$ and $b\in V_j$, $ab\in E(F)$ if and only if $v_iv_j\in E(G)$. Given two disjoint subsets $A,B\subseteq V(G)$, we will let $G[A]$ denote the subgraph of $G$ induced on the vertex set $A$, and let $G[A,B]$ denote the bipartite graph induced in $G$ with parts $A$ and $B$. Recall that the eccentricity of a vertex $x$ in a connected graph $F$, denoted $\mathrm{ecc}_F(x)$ is defined as
\[
\mathrm{ecc}_F(x)=\max\{d_F(x,y)\mid y\in V(F)\}.
\]
Note that if a graph $H$ is connected, then there exists a vertex $x\in V(H)$ such that $\mathrm{ecc}_H(x)=\mathrm{diam}(H)$.

\subsection{Equivalent Formulation}\label{section equivalent formulation}

In this section, we provide a brief proof that conjectures \ref{conjecture Rysers first formulation} and \ref{conjecture Rysers second formulation} are equivalent.

First, let us assume Conjecture \ref{conjecture Rysers first formulation} is true. Fix a graph $G$ with an $r$-coloring of the edges of $G$. Let us define a hypergraph $H$ where the vertex set of $H$ is the collection of monochromatic components of $G$, and the edges of $H$ are maximal collections of components with non-empty vertex intersection. Since two components of the same color do not intersect, $H$ is $r$-partite. Note that any monochromatic component cover of $G$ corresponds to a vertex cover of $H$, and thus $\mathrm{tc}_r(G)\leq \tau(H)$. Furthermore, given a matching $M\subseteq E(H)$, $M=\{e_1,e_2,\dots,e_k\}$ for each edge $e_i$ in $M$, there is at least one vertex $v_i\in V(G)$ such that $v_i$ is in the intersection of all components that are contained in $e_i$. Consider two vertices, $v_i$ and $v_j$, $i \neq j$, in the intersection of all components that are contained in $e_i$ and $e_j$, respectively. If $v_iv_j\in E(G)$, say of color $c$, then in order for $e_i$ and $e_j$ to be disjoint, one of them must not contain a component of color $c$. However, this contradicts that all edges of $H$ are maximal collections of components. Thus, $\{v_1,v_2,\dots,v_k\}\subseteq V(G)$ is an independent set and so $\nu(H)\leq \alpha(G)$. Then Conjecture~\ref{conjecture Rysers first formulation} implies that
\[
\mathrm{tc}_r(G)\leq \tau(H)\leq (r-1)\nu(H)\leq (r-1)\alpha(G),
\]
and thus Conjecture~\ref{conjecture Rysers first formulation} implies Conjecture~\ref{conjecture Rysers second formulation}.

Now let us assume Conjecture~\ref{conjecture Rysers second formulation} holds. Let $H$ be an $r$-partite hypergraph with partition $V(H)=V_1\cup V_2\cup\dots\cup V_r$. Define a graph $G$ with $V(G)=E(H)$. For each pair $u,v\in V(G)$, add the edge $uv$ to $E(G)$ if $u\cap v\neq \emptyset$ as edges of $H$. Color this edge color $i$, where $i$ is the smallest integer such that $u\cap v\cap V_i$ is non-empty. Given a monochromatic component cover of $G$, say $G_1,G_2,\dots,G_k\subseteq G$, where $G_i$ is color $c_i$, there is a unique vertex $v_i$ of $H$ in $V_{c_i}$ such that for all $u\in V(G_i)$, $u\cap V_{c_i}=\{v_i\}$. Thus the vertex $v_i$ in $H$ covers every edge corresponding to a vertex in $V(G_i)$. Since $\bigcup_{i=1}^k V(G_i)=V(G)=E(H)$, this implies that $\{v_1,v_2,\dots,v_k\}$ is a vertex cover of $H$, and thus $\tau(H)\leq \mathrm{tc}_r(G)$. Furthermore, given non-adjacent vertices $u,v\in V(G)$, this implies that as edges of $H$, $u\cap v=\emptyset$, and thus an independent set of $G$ corresponds to a matching of $H$, and so $\alpha(G)\leq \nu(H)$. Thus, Conjecture~\ref{conjecture Rysers second formulation} gives us that
\[
\tau(H)\leq \mathrm{tc}_r(G)\leq (r-1)\alpha(G)\leq (r-1)\nu(H),
\]
so Conjecture~\ref{conjecture Rysers first formulation} holds.

\subsection{For Any Complete Multipartite Graph $G$, $tc_2(G) \le 2$ }\label{tc}

Note that adding edges can only decrease the tree cover number, so it suffices to show this for complete bipartite graphs. Let $G = (A,B,E)$ be a complete bipartite graph, and fix $v \in A$. Let $A'$ and $B'$ be the sets of vertices in $A$ and $B$, respectively, which are covered by the maximum red component containing $v$. If $A' = A$, then we can cover the graph with the maximum red and blue components containing $v$. If instead $B' = B$ then all vertices in $A \setminus A'$ are complete blue to $B$ and therefore may be covered with a blue connected component. That blue component along with the red component containing $v$ forms the desired cover. Otherwise, we may cover the graph with $2$ connected blue components,  $G[(A \setminus A') \cup B']$ and $G[A' \cup (B \setminus B')]$, both of which induce a complete blue bipartite graph.

\section{Upper Bound - Proof of Theorem \ref{theorem tripartite upper bound}}\label{sec2}
In order to prove Theorem \ref{theorem tripartite upper bound}, we will first show in Section \ref{section reduction} that any counterexample to the theorem must have some very specific structure, and then in Section \ref{section the final case} we explore this structure and show that we can find a suitable monochromatic subgraph cover.

First note that if our complete multipartite graph has a part of size $1$, then this vertex is the center of a blue star and a red star (one of these may be trivial), which cover the entire graph, so we have a cover with two subgraphs of diameter at most 2 in this case. Now, let $a,b,c\geq 2$ be fixed, and assume to the contrary that Theorem \ref{theorem tripartite upper bound} does not hold for $G:=K_{a,b,c}$. Let $A,B$ and $C$ be the partite sets of $G$ with $|A|=a$, $|B|=b$ and $|C|=c$. Let $\chi:E(G)\to \{\text{red},\text{blue}\}$ be a $2$-coloring of the edges of $G$ where $\chi$ has the property that there do not exist two monochromatic subgraphs of $G$ of diameter at most $3$ that cover $V(G)$. Let $G_{red}$ and $G_{blue}$ be the spanning subgraphs of $G$ containing all the red edges and all the blue edges respectively, and define $N_{red}(x) := N_{G_{red}}(x), N_{blue}(x) := N_{G_{blue}}(x), d_{red}(x) := d_{G_{red}}(x),$ and $d_{red}(x) := d_{G_{red}}(x)$.

Throughout this section, we will use the following result:

\begin{lemma}[Lemma 4.18(P1) in \cite{Grace}]\label{lemma prev paper dom vtx}
	If all edges incident to a single vertex in a two-edge coloring of a complete bipartite graph are all colored the same, then there is a monochromatic subgraph cover with a star and a double star.
\end{lemma}

As noted in \cite{Grace}, the only case in which we do not know that a 2-colored complete bipartite graph has a monochromatic cover of diameter at most 3 occurs when both the red and blue spanning subgraphs are connected and of diameter exactly 4. This implies the following:

\begin{proposition}(\cite{Grace})\label{proposition diam4}
	Both graphs $G_{red}$ and $G_{blue}$ have diameter $4$.
\end{proposition}

Throughout the rest of this section, we will fix a vertex $v\in V(G)$ such that $\mathrm{ecc}_{G_{red}}(v) = 4$. We may assume without loss of generality that $v\in A$. 

We now will partition $A$, $B$ and $C$ based on red-distances from the vertex $v$. For $1\leq i\leq 4$, let $A_i\subseteq A$, $B_i\subseteq B$ and $C_i\subseteq C$ be the sets of vertices that are red-distance exactly $i$ from $v$ in each of our three partite sets. This gives us the following partitions:
\begin{itemize}
	\item $A = \{v\} \cup A_2 \cup A_3 \cup A_{4}$,
	\item $B = B_1 \cup B_2 \cup B_3 \cup B_{4}$, and
	\item $C = C_1 \cup C_2 \cup C_3 \cup C_{4}$,
\end{itemize}
where $A_1$ is omitted since $A_1=\emptyset$ trivially.
%%%%%%%%%%%%%%%%%%%%%%%%%%%%%%%%%%%%%%%%%%%%%%%%%%%%%%%%%%%%%%%%%%%%%%%%%%%%%%%%%%%%%
%%%%%%%%%%%%%%%%%%%%%%%%%%%%%%%%%%%%%%%%%%%%%%%%%%%%%%%%%%%%%%%%%%%%%%%%%%%%%%%%%%%%%
%%%%%%%%%%%%%%%%%%%%%%%%%%%%%%%%%%%%%%%%%%%%%%%%%%%%%%%%%%%%%%%%%%%%%%%%%%%%%%%%%%%%%
%%%%%%%%%%%%%%%%%%%%%%%%%%%%%%%%%%%%%%%%%%%%%%%%%%%%%%%%%%%%%%%%%%%%%%%%%%%%%%%%%%%%%
\subsection{Reducing the Problem}\label{section reduction}

Many of the sets in the above partitions of $A$, $B$ and $C$ may be empty. The goal of this section is to reduce the problem down to a case where we know exactly which sets are empty and which are non-empty. More specifically, we will show that we may focus on the case when $\{v\}, A_3,B_2,B_4,C_2$ and $C_4$ are all empty, while all the other sets in our partitions are non-empty.

\begin{claim}\label{claim sets at distance 1}
	$B_1, C_1\neq \emptyset$.
\end{claim}

\begin{proof}
	First note by Lemma \ref{lemma prev paper dom vtx} that no vertex of $G$ dominates an entire partite set in a single color, otherwise the bipartite subgraph between the dominated partite set and the rest of the graph would have a cover with a star and a double star. Therefore, there must exist a vertex in $B$ that is adjacent to $v$ in red, and also a vertex in $C$ that is adjacent to $v$ in red, giving us that $B_1\neq \emptyset$ and $C_1\neq \emptyset$.
\end{proof}

We now deal with all vertices at red-distance 4 from $v$.

\begin{claim}\label{claim sets at distance 4}
	$A_{4}\neq \emptyset$ and $B_{4}= C_{4}=\emptyset$.
\end{claim}

\begin{proof}
	First suppose $B_4 \neq \emptyset$, and let $u_4 \in B_4$. The existence of $u_4$ immediately implies that there exists a vertex $u_3 \in A_3 \cup C_3$. Now by claim \ref{claim sets at distance 1} we know that $B_1 \ne \emptyset$, so let $u_1 \in B_1$. Note that the edges $vu_4$ and $u_1u_3$ are both blue since $d_{red}(v,u_4)= 4$ and $d_{red}(u_1,u_3)\geq 2$. Let $S_1$ be the largest blue double star with centers $v$ and $u_4$ and let $S_2$ be the largest blue double star with centers $u_1$ and $u_3$. 
	
	We will show that $S_1$ and $S_2$ form a monochromatic subgraph cover of diameter at most $3$, which will give us a contradiction. Indeed, every vertex in $B_1$ is a blue neighbor of $u_3$, while $B\setminus B_1\subseteq N_{blue}(v)$, so $B\subseteq V(S_1)\cup V(S_2)$. Furthermore, every vertex in $A_3\cup A_{4} \cup C_3\cup C_{4}$ is blue-adjacent to $u_1$, while every vertex in $\{v\}\cup A_2\cup C_1\cup C_2$ is blue-adjacent to $u_4$, giving us that $A,C\subseteq V(S_1)\cup V(S_2)$. This gives us the desired contradiction, so we must have that $B_{4}=\emptyset$. By symmetry $C_{4}=\emptyset$. Since $v$ has eccentricity at least $4$ in $G_{red}$, we must have $A_{4}\neq \emptyset$, completing the proof.
\end{proof}

Next, we consider $B_3$ and $C_3$. 

\begin{claim}\label{claim sets at distance 3}
	$B_3, C_3 \ne \emptyset$.
\end{claim}

\begin{proof}
	By Claim \ref{claim sets at distance 4}, we must have a vertex $x\in A_{4}$. Note that $N_{red}(x)\subseteq B_3\cup C_3$. The contrapositive of Lemma \ref{lemma prev paper dom vtx} applied to $G[A \cup C, B]$ implies that $N_{red}(x)\cap B_3\neq \emptyset$. In particular, $B_3\neq \emptyset$, and by a symmetric argument, we can also conclude that $C_3\neq \emptyset$. 
\end{proof}

Before we show that $A_3=\emptyset$, we deal with all vertices at red-distance 2 from $v$.

\begin{claim}\label{claim sets at distance 2}
	$A_2\neq \emptyset$ and $B_2=C_2=\emptyset$.
\end{claim}

\begin{proof}
	First assume that $B_2\neq \emptyset$. Note that by definition, for all distinct $X,Y \in \{A,B,C\}$ and $i,j \in \{1,2,3,4\}$ with $|i-j| \ge 2$, $G[X_i,Y_j]$ is complete blue. Additionally, if $X \neq A$, and $i \in \{2,3,4\}$, $G[\{v\},X_i]$ is complete blue. These blue complete bipartite graphs are enough to guarantee that $H:=G_{blue}[V(G)\setminus(A_2\cup A_3)]$, as seen in Figure~\ref{H} has diameter $3$, regardless of whether $C_2$ is empty or not, and further that every vertex in $B_2$ has eccentricity $2$ in $H$.

	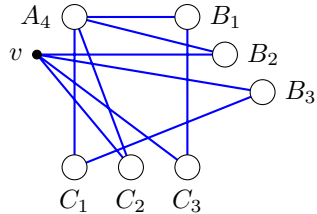
\begin{figure}[h]
		\centering
		\begin{tikzpicture}[scale=1]
				\draw[color=blue, thick] (0.5,1.5)--(1.75,0)--(1,2)--(3,1.5)--(0.5,1.5)--(2.5,0)--(2.5,2)--(1,2)--(1,0)--(3.5,1)--(0.5,1.5);
				
			\draw[color=black] (0.5,1.5) node[draw,shape=circle,fill=black,scale=0.35,label=left:{$v$}] {};
			\draw[color=black] (1,2) node[draw,shape=circle,fill=white,scale=1,label=left:{$A_4$}] {};
			\draw[color=black] (2.5,2) node[draw,shape=circle,fill=white,scale=1,label=right:{$B_1$}] {};
			\draw[color=black] (3,1.5) node[draw,shape=circle,fill=white,scale=1,label=right:{$B_2$}] {};
			\draw[color=black] (3.5,1) node[draw,shape=circle,fill=white,scale=1,label=right:{$B_3$}] {};
			\draw[color=black] (1,0) node[draw,shape=circle,fill=white,scale=1,label=below:{$C_1$}] {};
			\draw[color=black] (1.75,0) node[draw,shape=circle,fill=white,scale=1,label=below:{$C_2$}] {};
			\draw[color=black] (2.5,0) node[draw,shape=circle,fill=white,scale=1,label=below:{$C_3$}] {};

		\end{tikzpicture}
		\caption{The blue subgraph $H$, which has diameter 3}
		\label{H}
	\end{figure}

	Let $x\in B_2$ and define $H':=G_{blue}[V(H)\cup N_{blue}(x)]$. Since $x$ has eccentricity $2$ in $H$, $H'$ is a blue subgraph of $G$ of diameter $3$. Let $S$ be the largest red star in $G$ centered at $x$. Every vertex in $B \cup C$ is contained in $H'$, and $A \setminus V(H') \subseteq N_{red}(x)$. Thus $H'$ and $S$ give us a monochromatic subgraph cover of diameter at most $3$, a contradiction. Hence $B_2=\emptyset$, and a symmetric argument shows that $C_2=\emptyset$. Since $B_3$ is non-empty by Claim \ref{claim sets at distance 3}, there must exist a vertex at distance exactly $2$ from $v$, so we must have that $A_2$ is non-empty, finishing the proof of the claim.
\end{proof}

The final set to consider is $A_3$.

\begin{claim}
	$A_3 = \emptyset$.
\end{claim}

\begin{proof}
	 Any vertex in $A_3$ must have a red neighbor in $B_2 \cup C_2$ in order to have red-distance 3 from $v$. By Claim \ref{claim sets at distance 2}, $B_2$ and $C_2$ are empty. Therefore $A_3$ must also be empty.
\end{proof}

Thus, we have shown that the sets $\{v\}$, $A_2$, $A_4$, $B_1$, $B_3$, $C_1$, $C_3$ are all non-empty, while all the other sets in our partitions of $A,B$ and $C$ are empty.

%%%%%%%%%%%%%%%%%%%%%%%%%%%%%%%%%%%%%%%%%%%%%%%%%%%%%%%%%%%%%%%%%%%%%%%%%%%%%%%%%%%%%
%%%%%%%%%%%%%%%%%%%%%%%%%%%%%%%%%%%%%%%%%%%%%%%%%%%%%%%%%%%%%%%%%%%%%%%%%%%%%%%%%%%%%
%%%%%%%%%%%%%%%%%%%%%%%%%%%%%%%%%%%%%%%%%%%%%%%%%%%%%%%%%%%%%%%%%%%%%%%%%%%%%%%%%%%%%
%%%%%%%%%%%%%%%%%%%%%%%%%%%%%%%%%%%%%%%%%%%%%%%%%%%%%%%%%%%%%%%%%%%%%%%%%%%%%%%%%%%%%
\subsection{The Final Case of Theorem \ref{theorem tripartite upper bound}}\label{section the final case}
Based on Section \ref{section reduction}, we may assume that $A=\{v\}\cup A_2\cup A_{4}$, $B=B_1\cup B_3$ and $C=C_1\cup C_3$, and furthermore each of these sets are non-empty. We will now explore this particular case and show that there must exist a monochromatic subgraph cover of diameter at most $3$, which will conclude the proof of Theorem \ref{theorem tripartite upper bound}.

Note that, similar to the proof of Claim \ref{claim sets at distance 2}, $G[\{v\},B_3]$, $G[B_3,C_1]$, $G[C_1,A_{4}]$, $G[A_{4},B_1]$, $G[B_1,C_3]$ and $G[C_3,\{v\}]$ are all blue complete bipartite graphs due to the red-distance from $v$ of each set. This implies that $G':=G[V(G)\setminus A_2]$ has a spanning blue $C_6$-blowup, call it $C_6^+$, and thus has diameter at most $3$. Thus, our main goal of this section will be to show that the vertices in $A_2$ can either be added to $C_6^+$ without increasing the diameter, or can be included in a red subgraph of diameter $3$.

Towards this, we first explore which edges are red in $G'$. By definition, we have that $G[\{v\},B_1\cup C_1]$ is a red star. Our next claim gives us a collection of red edges in $G'$ that will help form a large red subgraph.

\begin{claim}\label{claim red complete BC}
	The red subgraph $G_{red}[B_1,C_1]$, is a complete bipartite graph.
\end{claim}

\begin{proof}
	Let $x\in B_1$ and $y\in C_1$. Assume to the contrary that the edge $xy$ is blue. Then $x$ has eccentricity $2$ in $C_6^++xy$, so $G^*:=G_{blue}[V(C_6^+)\cup N_{blue}(x)]$ has diameter at most $3$. Let $S$ be the largest red star with center $x$. Notice that $G^*$ and $S$ cover $G$ since the only vertices that are not in $G^*$ are in $A_2\setminus N_{blue}(x)$, which is contained in $V(S)$. This gives us a monochromatic subgraph cover of diameter at most $3$, a contradiction. Thus every edge in $G[B_1,C_1]$ is red.
\end{proof}

Note that using the same technique we can show that $G[B_3,C_3],$ $G[B_3,A_4],$ and $G[C_3,A_4]$ are complete red bipartite graphs as well, but only Claim \ref{claim red complete BC} is necessary to complete the proof of Theorem \ref{theorem tripartite upper bound}. We now define a partition of $A_2$ into two sets, one of which can be added to the blue subgraph containing $C_6^+$, and the other which can be covered with a red subgraph of diameter at most $3$. Let $A_{2,\text{red}}\subseteq A_2$ be the set of vertices, $x$, that satisfy at least one of the following properties:
\begin{itemize}
	\item[(P1)] $x$ has only red neighbors in $B_1$,\label{property B}
	\item[(P2)] $x$ has only red neighbors in $C_1$, or\label{propery C}
	\item[(P3)] $x$ has at least one red neighbor in each of $B_1$ and $C_1$.\label{Property B and C}
\end{itemize}
Let $A_{2,\text{blue}}:=A_2\setminus A_{2,\text{red}}$, and note that every vertex in $A_{2,\text{blue}}$ has at least one blue neighbor in each of $B_1$ and $C_1$, and also has only blue neighbors in at least one of $B_1$ or $C_1$. First we show that the vertices in $A_{2,\text{blue}}$ can be included in a blue subgraph containing the blue $C_6^+$.

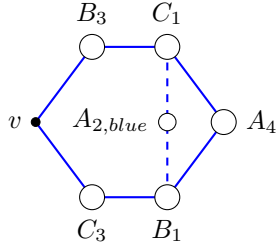
\begin{figure}[h]\label{blue}
	\centering
	\begin{tikzpicture}[scale=1]

		\draw[color=blue, thick] (-.25,1)--(0.5,0)--(1.5,0)--(2.25,1)--(1.5,2)--(0.5,2)--(-.25,1);
		\draw[color=blue,dashed, thick] (1.5,0)--(1.5,1)--(1.5,2);

		\draw[color=black] (-.25,1) node[draw,shape=circle,fill=black,scale=0.35,label=left:{$v$}] {};
		\draw[color=black] (0.5,0) node[draw,shape=circle,fill=white,scale=1,label=below:{$C_3$}] {};
		\draw[color=black] (0.5,2) node[draw,shape=circle,fill=white,scale=1,label=above:{$B_3$}] {};
		\draw[color=black] (1.5,0) node[draw,shape=circle,fill=white,scale=1,label=below:{$B_1$}] {};
		\draw[color=black] (1.5,1) node[draw,shape=circle,fill=white,scale=.7,label=left:{$A_{2,blue}$}] {};
		\draw[color=black] (1.5,2) node[draw,shape=circle,fill=white,scale=1,label=above:{$C_1$}] {};
		\draw[color=black] (2.25,1) node[draw,shape=circle,fill=white,scale=1,label=right:{$A_4$}] {};
		
	\end{tikzpicture}
	\caption{The blue subgraph $G_{blue}[V(C_6^+)\cup A_{2,\text{blue}}]$. Solid edges represent blue complete bipartite graphs. Dashed edges represent that every vertex in $A_{2,blue}$ is complete blue to one of $B_1$ and $C_1$, and has at least one blue neighbor in the other set.}
\end{figure}
    %A potential shorter (but less explicit) version to the ``Dashed edges..." part of the above caption: Dashed edges represent that no vertex in $A_{2,blue}$ satisfies any of (P1)-(P3).

\begin{claim}\label{claim blue subgraph diameter}
	The blue subgraph $G_{blue}[V(C_6^+)\cup A_{2,\text{blue}}]$ has diameter at most $3$.
\end{claim}

\begin{proof}
	We need only consider pairs of vertices $x$ and $y$ with at least one vertex in $A_{2,\text{blue}}$, say without loss of generality $x \in A_{2,\text{blue}}$. If $y$ is also in $A_{2,\text{blue}}$, then since $x$ has at least one neighbor in each of $B_1$ and $C_1$, and $y$ is adjacent to every vertex in either $B_1$ or $C_1$, $x$ and $y$ are at distance $2$ in $G_{blue}[V(C_6^+)\cup A_{2,\text{blue}}]$.
	
	Now, consider a pair with $x\in A_{2,\text{blue}}$ and $y\in V(C_6^+)$. Assume without loss of generality that $x$ has only blue neighbors in $B_1$, and let $z\in C_1$ be a blue neighbor of $x$. If $y\in B_3$, then $(x,z,y)$ is a blue path of length $2$. If $y\not\in B_3$, then $y$ is distance at most $2$ from any vertex in $B_1$, so distance at most $3$ from $x$, so in either case, $x$ and $y$ are at distance at most $3$, so $G_{blue}[V(C_6^+)\cup A_{2,\text{blue}}]$ has diameter at most $3$.
\end{proof}

The final step in the proof is to show that we can cover $A_{2,\text{red}}$ with a red subgraph.

\begin{figure}[h]\label{red}
	\centering
	\begin{tikzpicture}[scale=1]
		
		\draw[color=red, thick] (0,1)--(0,0)--(2,0)--(2,1);
		\draw[color=red,dotted, thick] (0,0)--(1,1)--(2,0);
		
		\draw[color=black] (0,0) node[draw,shape=circle,fill=white,scale=1,label=below:{$B_1$}] {};
		\draw[color=black] (2,0) node[draw,shape=circle,fill=white,scale=1,label=below:{$C_1$}] {};
		\draw[color=black] (0,1) node[draw,shape=circle,fill=white,scale=.7,label=above:{$X_1$}] {};
		\draw[color=black] (1,1) node[draw,shape=circle,fill=white,scale=.7,label=above:{$X_3$}] {};
		\draw[color=black] (2,1) node[draw,shape=circle,fill=white,scale=.7,label=above:{$X_2$}] {};

	\end{tikzpicture}
	\caption{The red subgraph $G_{red}[B_1\cup C_1\cup A_{2,\text{red}}]$. Solid edges represent complete bipartite graphs while dotted edges represent that vertices in $X_3$ have at least one red neighbor in each set $B_1$ and $C_1$.}
\end{figure}
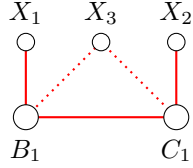

\begin{claim}\label{claim red subgraph diameter}
	The red subgraph $G_{red}[B_1\cup C_1\cup A_{2,\text{red}}]$ has diameter at most $3$. 
\end{claim}

\begin{proof}
	By Claim \ref{claim red complete BC}, $G_{red}[B_1,C_1]$ is complete bipartite. Let $X_i\subseteq A_2$ be the set of vertices satisfying property (P$i$) for $1\leq i\leq 3$. Then note that $G_{red}[B_1\cup C_1\cup X_1\cup X_2]$ has a spanning red $P_4$-blowup, call it $P_4^+$, which has diameter $3$. Thus, to show that $G_{red}[B_1\cup C_1\cup A_{2,\text{red}}]$ has the desired diameter, we only need to consider distances between pairs of vertices with at least one vertex in $X_3$. 
	
	Let $x\in X_3$ and $y\in B_1\cup C_1\cup A_{2,\text{red}}$. Since $\mathrm{ecc}_{P_4^+}(u)=2$ for any vertex $u\in B_1\cup C_1$, and $x$ has a red neighbor in $B_1\cup C_1$, $d_{red}(x,y) \leq 3$ if $y \in B_1\cup C_1\cup X_1\cup X_2$. If instead $y\in X_3$, then let $x'\in B_1$ be a red neighbor of $x$, and $y'\in C_1$ be a red neighbor of $y$. The path $(x,x',y',y)$ is a red path of length $3$, so $d_{red}(x,y) \leq 3$, finishing the proof.
\end{proof}

This completes the proof of Theorem \ref{theorem tripartite upper bound}; in Section \ref{section reduction}, we reduced the proof down to the case when $\{v\}, A_2, A_4, B_1, B_3, C_1$ and $C_3$ are the only non-empty sets in our original partition, and then via Claim \ref{claim blue subgraph diameter} and Claim \ref{claim red subgraph diameter}, we show that in this final case, we have a monochromatic subgraph cover using two subgraphs of diameter at most $3$.

\section{Determining $D_2^2(G)$ exactly for complete tripartite graphs and others}\label{sec3}

In this section we provide a complete classification of $D_2^2(G)$ for all complete tripartite graphs, $G$, as well as prove some results towards a classification of $D_2^2(G)$ for any complete multipartite graph $G$. In light of Theorem~\ref{theorem tripartite upper bound}, as the addition of edges can only decrease the value of $D_2^2(G)$, $D_2^2(G)\leq 3$ for all complete multipartite graphs with at least three parts. In Theorem~\ref{theorem most complete multipartite graphs are 3} we prove that a particular complete multipartite graph $G$ has $D_2^2(G)=3$, and then in Proposition~\ref{proposition monotonicity}, we prove that if we add a vertex to an existing part in a complete multipartite graph $G$, the diameter cover number cannot decrease. For any fixed $k$, this gives us that there are finitely many complete $k$ partite graphs with diameter cover number $2$ or less. In the case of complete tripartite graphs, we then classify the remaining graphs via direct analysis.

Our first result in this section shows that for each $k\geq 2$, a complete $(k+1)$-partite graph with no part of size $1$ and one large part will have diameter cover number $3$. Let $K_{a,b(k)}$ denote the complete $(k+1)$-partite graph with one part of size $a$ and $k$ parts of size $b$.

\begin{theorem}\label{theorem most complete multipartite graphs are 3}
	For all $k\geq 2$,
	\[
	D_2^2(K_{2k+1,2(k)})=3.
	\]
\end{theorem}

\begin{proof}
	The upper bound follows from Theorem~\ref{theorem tripartite upper bound} and the fact that $K_{2k+1,2k-2,2}$ is a spanning subgraph of $K_{2k+1,2(k)}$. Now let us focus on the lower bound.
	
	Let $G=K_{2k+1,2(k)}$, Let $A=\{a_1,a_2,\dots,a_{2k},c\}$ be the part of $G$ of size $2k+1$, and let $B=\{b_1,b_2,\dots,b_{2k}\}$ denote the remaining vertices of $G$, where $\{b_{2i-1},b_{2i}\}$ is a part of size $2$ for all $1\leq i\leq k$. Color the edges $a_ib_i$, $cb_i$, and $b_ib_j$ blue for all $1\leq i,j\leq 2k$ (if $b_ib_j$ is not an edge, we do not assign it a color), and color all remaining edges of $G$ red.
	
	We claim that under this coloring, $G$ does not have a monochromatic cover with two graphs of diameter at most $2$. To see this, let us assume to the contrary that there exists such a cover. First note that since $c$ is not incident to any red edges, there must be a blue subgraph containing $c$. Furthermore, the blue edges incident to the $a_i$'s form a matching, so any blue subgraph of diameter $2$ can contain at most one of the $a_i$'s. and since there are $2k\geq 4$ such vertices $a_i$, there must be a red subgraph that contains all but at most one of the $a_i$'s. This red subgraph must contain at least one of the $b_i$'s, otherwise it would contain no edges, so without loss of generality, we can assume that $b_1$ is in the red subgraph. Since $a_1$ and $b_1$ are distance $3$ from each other in red, $a_1$ cannot be in the red subgraph, and so the red subgraph contains the vertices in $\{a_2,a_3,\dots,a_{2k}\}$. Now, note that $b_2$ is distance $3$ from $a_2$ in red, and distance $3$ from $a_1$ in blue, so $b_2$ cannot be in either the red or blue subgraph, a contradiction. Thus $D_2^2(G)\geq 3$.
\end{proof}

	\begin{figure}
		\begin{center}
		\begin{tabular}{cc}
			\begin{tikzpicture}[scale=.9]
					\draw[thick,  color=blue] (0,3)--(-.5,.5)--(2,3)--(0,0)--(1,3);
				\draw[thick,  color=blue] (3,3)--(4,0)--(2,3)--(4.5,.5)--(4,3);
				\draw[thick,  color=blue] (-.5,.5)--(4,0)--(0,0)--(4.5,.5)--(-.5,.5);
				
				\draw[color=black] (0,3) node[draw,shape=circle,fill=black,scale=0.5] {};
				\draw[color=black] (1,3) node[draw,shape=circle,fill=black,scale=0.5] {};
				\draw[color=black] (2,3) node[draw,shape=circle,fill=black,scale=0.5] {};
				\draw[color=black] (3,3) node[draw,shape=circle,fill=black,scale=0.5] {};
				\draw[color=black] (4,3) node[draw,shape=circle,fill=black,scale=0.5] {};
				\draw[color=black] (-.5,.5) node[draw,shape=circle,fill=black,scale=0.5] {};
				\draw[color=black] (0,0) node[draw,shape=circle,fill=black,scale=0.5] {};
				\draw[color=black] (4,0) node[draw,shape=circle,fill=black,scale=0.5] {};
				\draw[color=black] (4.5,.5) node[draw,shape=circle,fill=black,scale=0.5] {};
				
				\draw (0,3.25) node {$a_1$};
				\draw (1,3.25) node {$a_2$};
				\draw (2,3.25) node {$c$};
				\draw (3,3.25) node {$a_3$};
				\draw (4,3.25) node {$a_4$};
				
				\draw (-.8,.5) node {$b_1$};
				\draw (-.3,0) node {$b_2$};
				
				\draw (4.35,0) node {$b_3$};
				\draw (4.85,.5) node {$b_4$};
			\end{tikzpicture}&	\begin{tikzpicture}[scale=.9]
				
				\draw[thick, color=red] (-.5,.5)--(1,3)--(4,0)--(0,3)--(0,0)--(3,3)--(-.5,.5)--(4,3)--(0,0);
				\draw[thick, color=red] (0,3)--(4.5,.5)--(1,3);
				\draw[thick, color=red] (3,3)--(4.5,.5);
				\draw[thick, color=red] (4,3)--(4,0);
				
				\draw[color=black] (0,3) node[draw,shape=circle,fill=black,scale=0.5] {};
				\draw[color=black] (1,3) node[draw,shape=circle,fill=black,scale=0.5] {};
				\draw[color=black] (2,3) node[draw,shape=circle,fill=black,scale=0.5] {};
				\draw[color=black] (3,3) node[draw,shape=circle,fill=black,scale=0.5] {};
				\draw[color=black] (4,3) node[draw,shape=circle,fill=black,scale=0.5] {};
				\draw[color=black] (-.5,.5) node[draw,shape=circle,fill=black,scale=0.5] {};
				\draw[color=black] (0,0) node[draw,shape=circle,fill=black,scale=0.5] {};
				\draw[color=black] (4,0) node[draw,shape=circle,fill=black,scale=0.5] {};
				\draw[color=black] (4.5,.5) node[draw,shape=circle,fill=black,scale=0.5] {};
				
				\draw (0,3.25) node {$a_1$};
				\draw (1,3.25) node {$a_2$};
				\draw (2,3.25) node {$c$};
				\draw (3,3.25) node {$a_3$};
				\draw (4,3.25) node {$a_4$};
				
				\draw (-.8,.5) node {$b_1$};
				\draw (-.3,0) node {$b_2$};
				
				\draw (4.35,0) node {$b_3$};
				\draw (4.85,.5) node {$b_4$};
			\end{tikzpicture}
		\end{tabular}
		\caption{A coloring of $K_{5,2,2}$ that does not admit a monochromatic cover with two diameter $2$ subgraphs}
		\label{3 coloring}
	\end{center}
	\end{figure}
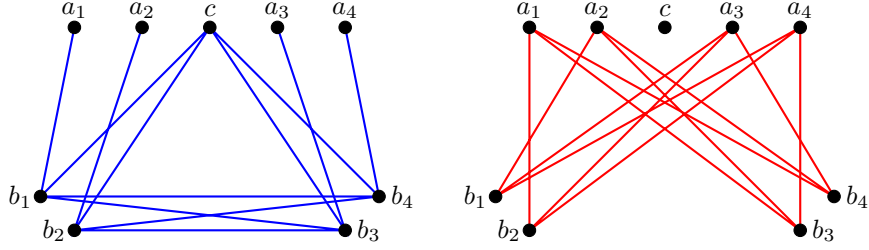

We now prove that adding vertices to a part in a complete multipartite graph does not decrease the diameter cover number.

\begin{proposition}\label{proposition monotonicity}
	Let $k\geq 2$, $b_1\geq a_1,b_2\geq a_2,\dots,b_k\geq a_k$ be positive integers. If $D_2^2(K_{a_1,a_2,\dots,a_k})\geq 3$, then $D_2^2(K_{b_1,b_2,\dots,b_k})\geq 3$. 
\end{proposition}

\begin{proof}
	Note that it will suffice to prove that $D_2^2(K_{a_1+1,a_2,\dots,a_k})\geq 3$. Let $G=K_{a_1+1,a_2,\dots,a_k}$, and let $G'$ be an induced subgraph of $G$ isomorphic to $K_{a_1,a_2,\dots,a_k}$. By assumption, there exists a $2$-coloring $c:E(G')\to \{\text{red},\text{blue}\}$ of the edges of $G'$ such that there is no cover of $G'$ using two monochromatic subgraphs, each with diameter at most $2$. Let the edges of $G'$ be colored according to this $2$-coloring.
	
	Let $x\in V(G')$ be a vertex in the part of size $a_1$. Let $y$ be the vertex in $V(G)\setminus V(G')$. For each vertex $v\in V(G)$ such that $vy\in E(G)$, assign the color $c(vx)$ to the edge $vy$. We claim that there is no cover of $G$ with two monochromatic subgraphs of diameter at most $2$. Assume to the contrary that there was such a cover, say with monochromatic subgraphs $G_1$ and $G_2$, where the edges of $G_i$ are colored $c_i\in \{\text{red},\text{blue}\}$ for $i=1,2$. Let $V_i=V(G_i)$ for $i=1,2$. For $i=1,2$, if $y\in V_i$, let $V_i'=(V_i\setminus\{y\})\cup \{x\}$, and if $y\not\in V_i$, let $V_i'=V_i$.  Let $G_i'$ denote the subgraph of $G'[V_i']$ consisting of only edges of color $c_i$.

	We claim that $G_i'$ is connected and has diameter $2$. Indeed, if $V_i'=V_i$, this follows immediately from our contrary assumption. If $V_i'=(V_i\setminus\{y\})\cup \{x\}$, any path in $G_i$ containing $y$ can be replaced with a path of the same length in $G_i'$ containing $x$, so $\mathrm{diam}(G_i')\leq \mathrm{diam}(G_i)\leq 2$. Finally, note that $V(G')=V_1'\cup V_2'$ since $V(G)=V_1\cup V_2$, and $V_i\setminus\{y\}\subseteq V_i'$. This contradicts the original assumption that under the coloring $c$, there was no cover of $G'$ with two monochromatic subgraphs with diameter at most $2$.
\end{proof}

By combining Theorem~\ref{theorem most complete multipartite graphs are 3} and Proposition~\ref{proposition monotonicity}, we see that given any $k\geq 3$, all but finitely many $k$-partite graphs $G$ with no part of size $1$ have $D_2^2(G)=2$. Let us turn our attention to small complete tripartite graphs with no part of size $1$. There are $10$ complete tripartite graphs with no part of size $1$ that our prior results do not imply a result for $D_2^2$, namely the graphs $K_{a,b,c}$ with $2\leq c\leq b\leq a\leq 4$. Fortunately, to determine which complete tripartite graphs require diameter $3$, we do not need to check all $10$ of these graphs, but instead we need to find the ``minimal'' ones that require diameter $3$ and the ``maximal'' ones that do not require diameter $3$, and then we can apply Proposition~\ref{proposition monotonicity} and its contrapositive to classify the rest.

\begin{proposition}\label{proposition small tripartite graphs}
	We have the following:
	\begin{itemize}
		\item $D_2^2(K_{4,3,2})=3$,
		\item $D_2^2(K_{4,2,2})=2$,
		\item $D_2^2(K_{3,3,3})=2$.
	\end{itemize}
\end{proposition}

As the proofs for Proposition~\ref{proposition small tripartite graphs} involve simple logic and extensive case work, we only include a proof of the first equality, below. These graphs are small enough that we were able to verify these bounds using a brute force computer program (see supplemental files on arXiv).

\begin{proof}[proof that $D_2^2(K_{4,3,2})=3$]
	Note that the upper bound on $D_2^2(K_{4,3,2})$ follows from Theorem~\ref{theorem tripartite upper bound}. For the lower bound, consider the coloring of $K_{4,3,2}$ shown in Figure \ref{432}. Assume towards a contradiction that there is some covering of the vertices with $2$ monochromatic diameter $2$ subgraphs. First we note that in the spanning blue subgraph vertices $v_2,v_3,$ and $v_6$ are pairwise distance at least $3$ from each other, and vertices $v_0,v_1,$ and $v_7$ are pairwise distance $3$ in the spanning red subgraph. Therefore, one of our subgraphs in our cover must be red, while the other must be blue.  Let $R$ be the set of vertices in the red subgraph, and $B$ be the set of vertices in the blue subgraph. We consider two cases based on which subgraphs contains $v_8$.
	
	\textbf{Case 1:} $v_8 \in R$. Since $v_2$ is distance $3$ from $v_8$ in red, we must have that $v_2 \in B\setminus R$. Furthermore, as $v_7$ is distance $3$ from $v_2$ in blue, $v_7 \in R\setminus B$. Since $v_1$ is distance 3 from $v_7$ in red, we have $v_1 \in B\setminus R$. Note that the only blue path of length $2$ from $v_1$ to $v_5$ is through $v_7$, so since $v_1\in B$, and $v_7\not\in B$, we have that $v_5\not\in B$. Similarly, the only red path of length $2$ from $v_7$ to $v_5$ is through $v_2$, so since $v_7\in R$ and $v_2\not\in R$, we have that $v_5\not\in R$, a contradiction.
	
	\textbf{Case 2:} $v_8 \in B\setminus R$. If $v_3\in B$, then $v_1$ and $v_4$ are in $R\setminus B$ since they are distance at least $3$ in blue from $v_3$, but the only red path of length $2$ from $v_1$ to $v_4$ uses $v_8$, which is not in $R$, a contradiction. Thus, $v_3\in R\setminus B$. As $v_8$ and $v_1$ are distance 3 in blue, we know that $v_1 \in R\setminus B$, but the only red path of length $2$ from $v_1$ to $v_3$ is through $v_8$, yielding another contradiction and completing the proof.
\end{proof}

	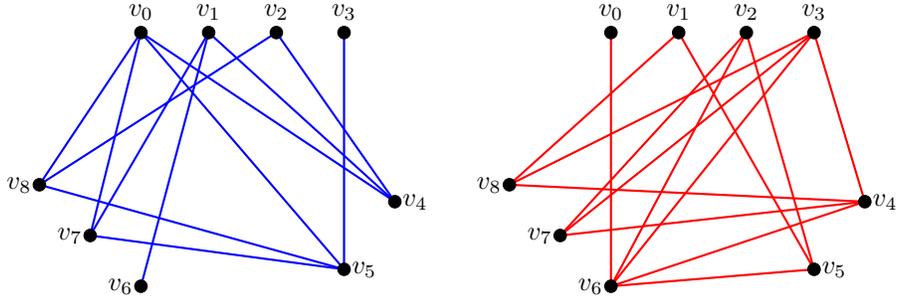
\begin{figure}
		\begin{center}
		\begin{tabular}{cc}
			\begin{tikzpicture}[scale=.9]
				\draw[thick,  color=blue] (2,3)--(0.5,0.75);
				\draw[thick,  color=blue] (2,3)--(1.25,0);
				\draw[thick,  color=blue] (2,3)--(5,-0.5);
				\draw[thick,  color=blue] (2,3)--(5.75,0.5);
				
				\draw[thick,  color=blue] (3,3)--(1.25,0);
				\draw[thick,  color=blue] (3,3)--(2,-0.75);
				\draw[thick,  color=blue] (3,3)--(5.75,0.5);
				
				\draw[thick,  color=blue] (4,3)--(0.5,0.75);
				\draw[thick,  color=blue] (4,3)--(5.75,0.5);
				
				\draw[thick,  color=blue] (5,3)--(5,-0.5);
				
				\draw[thick,  color=blue] (5,-0.5)--(0.5,0.75);
				\draw[thick,  color=blue] (5,-0.5)--(1.25,0);
				
				\draw[color=black] (2,3) node[draw,shape=circle,fill=black,scale=0.5] {};
				\draw[color=black] (3,3) node[draw,shape=circle,fill=black,scale=0.5] {};
				\draw[color=black] (4,3) node[draw,shape=circle,fill=black,scale=0.5] {};
				\draw[color=black] (5,3) node[draw,shape=circle,fill=black,scale=0.5] {};
				\draw[color=black] (0.5,0.75) node[draw,shape=circle,fill=black,scale=0.5] {};
				\draw[color=black] (1.25,0) node[draw,shape=circle,fill=black,scale=0.5] {};
				\draw[color=black] (2,-0.75) node[draw,shape=circle,fill=black,scale=0.5] {};
				\draw[color=black] (5,-0.5) node[draw,shape=circle,fill=black,scale=0.5] {};
				\draw[color=black] (5.75,0.5) node[draw,shape=circle,fill=black,scale=0.5] {};

				\draw (2,3.3) node {$v_0$};
				\draw (3,3.3) node {$v_1$};
				\draw (4,3.3) node {$v_2$};
				\draw (5,3.3) node {$v_3$};
				\draw (6.05,0.5) node{$v_4$};
				\draw (5.3,-0.5) node{$v_5$};
				\draw (1.7,-.75) node{$v_6$};
				\draw (0.95,0) node{$v_7$};
				\draw (0.2,0.75) node{$v_8$};
			\end{tikzpicture}&	\begin{tikzpicture}[scale=.9]
			
				\draw[thick,  color=red] (2,3)--(2,-0.75);
			
			\draw[thick,  color=red] (3,3)--(0.5,0.75);
			\draw[thick,  color=red] (3,3)--(5,-0.5);
			
			\draw[thick,  color=red] (4,3)--(2,-0.75);
			\draw[thick,  color=red] (4,3)--(1.25,-0);
			\draw[thick,  color=red] (4,3)--(5,-0.5);
			
			\draw[thick,  color=red] (5,3)--(2,-0.75);
			\draw[thick,  color=red] (5,3)--(1.25,-0);
			\draw[thick,  color=red] (5,3)--(0.5,0.75);
			\draw[thick,  color=red] (5,3)--(5.75,0.5);
			
			\draw[thick,  color=red] (5.75,0.5)--(2,-0.75);
			\draw[thick,  color=red] (5.75,0.5)--(1.25,0);
			\draw[thick,  color=red] (5.75,0.5)--(0.5,0.75);
			
			\draw[thick,  color=red] (5,-0.5)--(2,-0.75);
				
				\draw[color=black] (2,3) node[draw,shape=circle,fill=black,scale=0.5] {};
				\draw[color=black] (3,3) node[draw,shape=circle,fill=black,scale=0.5] {};
				\draw[color=black] (4,3) node[draw,shape=circle,fill=black,scale=0.5] {};
				\draw[color=black] (5,3) node[draw,shape=circle,fill=black,scale=0.5] {};
				\draw[color=black] (0.5,0.75) node[draw,shape=circle,fill=black,scale=0.5] {};
				\draw[color=black] (1.25,0) node[draw,shape=circle,fill=black,scale=0.5] {};
				\draw[color=black] (2,-0.75) node[draw,shape=circle,fill=black,scale=0.5] {};
				\draw[color=black] (5,-0.5) node[draw,shape=circle,fill=black,scale=0.5] {};
				\draw[color=black] (5.75,0.5) node[draw,shape=circle,fill=black,scale=0.5] {};
				
				\draw (2,3.3) node {$v_0$};
				\draw (3,3.3) node {$v_1$};
				\draw (4,3.3) node {$v_2$};
				\draw (5,3.3) node {$v_3$};
				\draw (6.05,0.5) node{$v_4$};
				\draw (5.3,-0.5) node{$v_5$};
				\draw (1.7,-.75) node{$v_6$};
				\draw (0.95,0) node{$v_7$};
				\draw (0.2,0.75) node{$v_8$};
			\end{tikzpicture}
		\end{tabular}
		\caption{A coloring of $K_{4,3,2}$ that does not admit a monochromatic cover with two diameter $2$ subgraphs}
		\label{432}
	\end{center}
	\end{figure}

Note that $D_2^2(K_3)=D_2^2(K_{2,1,1})=1$ since both graphs can be covered by two edges, but $D_2^2(K_{3,1,1})\geq 2$ since $K_{3,1,1}$ cannot be covered by two diameter $1$ subgraphs (even without regards to an edge coloring), and $D_2^2(K_{2,2,1})\geq 2$ since any cover of $K_{2,2,1}$ with two diameter $1$ subgraphs would necessarily need to contain at least one $K_3$, but if we color the edges incident with the vertex in the part of size $1$ red and all other edges blue, there is no monochromatic triangle. We remind the reader that any complete multipartite graph with a part of size $1$ has $D_2^2(G)\leq 2$ since a red and blue star centered at this vertex cover everything.

These results allow us to determine $D_2^2(G)$ for all complete tripartite graphs $G$. The results on complete tripartite graphs are summarized concisely below.

\begin{theorem}
	Let $G$ be a complete tripartite graph. 
	\begin{itemize}
		\item $D_2^2(G)=3$ if and only if $K_{5,2,2}\subseteq G$ or $K_{4,3,2}\subseteq G$,
		\item $D_2^2(G)=1$ if and only if $G=K_3$ or $G=K_{2,1,1}$, and
		\item $D_2^2(G)=2$ otherwise.
	\end{itemize}
\end{theorem}

\section{Progress towards Conjecture \ref{conjecture too many 2s}}\label{sec4}
In this section, we consider the family of complete multipartite graphs in which each partite set has size $2$. Recall that we denote such a graph with $k$ parts as $G_k$. 

\begin{problem}\label{problem too many 2s}
	Determine the value of $D_2^2(G_k)$.
\end{problem}

Using the same program that we checked small examples of complete tripartite graphs with, we found that $D_2^2(G_3)=D_2^2(G_4)=D_2^2(G_5)=2$. In attempts to determine whether this value is $2$ or $3$ for $k \geq 6$, we proved the following results. These results serve as properties of a minimal example which requires diameter $3$, if such a graph and coloring exists. Note that one difficulty in solving this problem is being unable to classify graphs of diameter $2$. In particular, we often attempt to build a cover using familiar graphs of diameter 2 that are relatively easy to find by hand, such as stars and $C_5$ blow-ups. However, there are many diameter 2 graphs which are more difficult to find by hand. In particular, almost all graphs are diameter 2, and further we have pathological examples such as the Petersen graph and the graph pictured in Figure \ref{beetle}, which was in fact used in multiple covers produced by a program we ran to check small cases.

\begin{figure}
	\begin{center}
	\begin{tikzpicture}[scale=2]
		\draw[color=black, thick] (.1,0)--(.9,0)--(0.5,-0.4)--(0.1,0)--(-0.1,-0.65)--(0.5,-1.2)--(.6,-.8)--(.5,-.4);
		\draw[color=black, thick] (.9,0)--(1.1,-0.65)--(.5,-1.2);
		
		\draw[color=black] (.1,0) node[draw,shape=circle,fill=black,scale=0.35] {};
		\draw[color=black] (.9,0) node[draw,shape=circle,fill=black,scale=0.35] {};
		\draw[color=black] (0.5,-0.4) node[draw,shape=circle,fill=black,scale=0.35] {};
		\draw[color=black] (-0.1,-0.65) node[draw,shape=circle,fill=black,scale=0.35] {};
		\draw[color=black] (1.1,-0.65) node[draw,shape=circle,fill=black,scale=0.35] {};
		\draw[color=black] (0.6,-0.8) node[draw,shape=circle,fill=black,scale=0.35] {};
		\draw[color=black] (0.5,-1.2) node[draw,shape=circle,fill=black,scale=0.35] {};
	\end{tikzpicture}
	\end{center}
	\caption{A diameter 2 graph used in some colorings of $K_{2,2,2,2,2}$.}
	\label{beetle}
\end{figure}
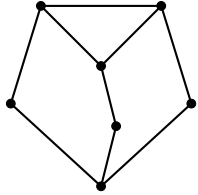

Throughout the section we will suppose that a 2-coloring of the edges of $G_k$ exists that requires a subgraph of diameter 3 in every cover, and state necessary properties. We fix such a 2-coloring of $G_k$ and let color 1 be blue and color 2 be red.

For any $v \in V(G_k)$, we will denote by $v'$ the unique vertex such that $vv' \notin E(G_k)$, and we call $v'$ the \textit{clone} of $v$. We begin with a result on the length 2 paths between clones.

\begin{property}\label{property distance 2}
	For each $v \in V(G_k)$, every possible color combination of length 2 $vv'$-paths exist. That is, clones are distance 2 from each other in both colors 1 and 2, as well as have length 2 paths which alternate colors in both orders.
\end{property}

\begin{proof}
	For $i,j \in [2]$, let $X_{i,j}$ denote the vertices which send color $i$ to $v$ and color $j$ to $v'$. Note that each of the four sets represent one of the four possible color combinations of length 2 $vv'$-paths. Suppose $X_{i,j} = \emptyset$ for some $i,j \in [2]$. Then the color $j$ star at $v$ and the color $i$ star at $j$ cover $V(G_k)$ and each have diameter 2. Hence $X_{i,j} \neq \emptyset$ for all $i,j \in [2]$.
\end{proof}

 Note that there must exist a vertex $x \in V(G_k)$ that has at least one vertex at distance at least 3 in blue, since otherwise the spanning blue subgraph is the desired cover of diameter at most 2. By the previous property, this other vertex is not $x'$. We will now partition the vertex set in terms of the distances in blue from $x$ and $x'$.

For $i,j \in [3]$, let $A_{i,j}$ denote the vertices that are distance $i$ in blue from $x$ and distance $j$ in blue from $x'$, with the convention that when $i$ or $j$ is 3, we include vertices of distance at least 3 in blue instead of exactly 3. Let $A_{i,*} = A_{i,1} \cup A_{i,2} \cup A_{i,3}$ and $A_{*,j} = A_{1,j} \cup A_{2,j} \cup A_{3,j}$.

Next, we eliminate three of these nine sets.

\begin{property}\label{property empty sets}
	$A_{2,3} \cup A_{3,2} \cup A_{3,3} = \emptyset$.
\end{property}

\begin{proof}
	First, suppose $A_{3,2} \cup A_{3,3} \neq \emptyset$ and let $y \in A_{3,2} \cup A_{3,3}$. Consider the red star at $x$ and the red star at $y$. The red star at $x$ covers $A_{2,*} \cup A_{3,*}$, and the red star at $y$ covers $\{x'\} \cup (A_{1,*}\setminus \{y'\})$. If $y' \notin A_{1,*}$, then this gives the desired cover with diameter at most 2.
	
	If $y' \in A_{1,1}$, then consider the red and blue stars at $x'$. Notice that the only red neighbors of $x'$ which do not also send a red edge to $x$ are $A_{1,2} \cup A_{1,3}$. However, all those vertices are red neighbors of $y$, so we can add $x$ to the red star at $x'$ while maintaining diameter 2. Thus we have the desired cover. 
	
	If $y' \in A_{1,2} \cup A_{1,3}$, then consider the red and blue stars at $y'$. Notice that the only vertex not covered by these two subgraphs is $y$. We will add $y$ to the red subgraph by also adding $x$. As long as either $A_{3,*}\setminus \{y\} \neq \emptyset$ or $y'$ has a red neighbor in $A_{2,*}$, this new red subgraph has diameter 2, and we have the desired cover. Otherwise, $y$ is the only vertex in $A_{3,*}$ and every vertex in $A_{2,*}$ sends a blue edge to $y'$. Hence the red star at $y$ and the blue star at $y'$ give the desired cover.
	
	Finally, we can switch the roles of $x$ and $x'$ in the above argument to show that if $A_{2,3} \neq \emptyset$, we get the desired cover.
\end{proof}

The previous two properties imply that the following sets must be nonempty.

\begin{corollary}\label{corollary nonempty sets}
	$A_{1,1} \neq \emptyset$, $A_{2,2} \neq \emptyset$, $A_{1,2} \cup A_{1,3} \neq \emptyset$, and $A_{2,1} \cup A_{3,1} \neq \emptyset$.
\end{corollary}

Note that by our choice of $x$ and Property \ref{property empty sets}, we have more specifically that $A_{3,1} \neq \emptyset$

Now we give the location of the clones of vertices in $A_{1,3}$ and $A_{3,1}$.

\begin{property}
	For any $y \in A_{1,3}$, $y' \in A_{2,1}$. Similarly, for any $z \in A_{3,1}$, $z' \in A_{1,2}$.
\end{property}

\begin{proof}
	Fix $y \in A_{1,3}$. First, we will show that $y' \in A_{2,1} \cup A_{3,1}$. Consider the blue star at $x$ along with the red $C_5$ blow up formed by $x$, $A_{2,2}$, $x'$, $y$, and $(A_{2,1} \cup A_{3,1})\setminus \{y'\}$. Note that $A_{2,2} \neq \emptyset$ and $A_{2,1} \cup A_{3,1} \neq \emptyset$ by Corollary \ref{corollary nonempty sets}. Thus if $y' \notin A_{2,1} \cup A_{3,1}$, those two monochromatic subgraphs form the desired cover.
	
	Now suppose $y' \in A_{3,1}$. Consider the red star at $y'$ along with the red $C_5$ blow up formed by $x$, $A_{2,2}$, $x'$, $y$, and $(A_{2,1} \cup A_{3,1})\setminus \{y'\}$. This gives the desired cover unless $(A_{2,1} \cup A_{3,1})\setminus \{y'\} = \emptyset$, in which case we replace the $C_5$ blow up with the red star at $x'$. Thus we have the desired cover.
	
	Therefore $y' \in A_{2,1}$. The second half of the property can be proven by switching the roles of $x$ and $x'$ in the above argument.
\end{proof}

\section{Concluding Results}
In this paper we were able to determine the diameter cover number exactly for complete tripartite graphs, and for some classes of complete multipartite graphs, whenever two colors are used and two subgraphs are allowed. It would however be interesting to have a more thorough idea of the diameter cover number for other complete multipartite graphs. One of the most tangible questions in this regard would be work towards Conjecture \ref{conjecture too many 2s}, restated here:

\begin{problem41}
	Determine the value of $D_2^2(G_k)$, where $G_k$ is the complete $k$-partite graph in which each part is size $2$.
\end{problem41}

Section \ref{sec4} covers some results we were able to show regarding the previous problem. Knowing this border case could also possibly help determine the diameter cover number for other complete multipartite graphs. In the other direction, one could also ask about complete bipartite graphs instead. 

\begin{problem}
	Determine the value of $D_2^2(K_{a,b})$.
\end{problem}

It has been shown that $D_2^2(K_{a,b}) \le 4$ in \cite{Grace}, however this is not known to be sharp, and one could work towards providing a complete classification as we did for complete tripartite graphs. 

One could also explore this question while allowing for more subgraphs. It is easy to see that if $G$ is complete multipartite, then $D_2^4(G)\leq 2$, since we can cover the entire graph with a red and blue star at a vertex in one part and another red and blue star at a vertex in another part. It is also not difficult to determine when $D_2^4(G)=1$, so the question with $4$ subgraphs is not particularly interesting. The behavior of $D_2^3(G)$ is less clear.

\begin{problem}
	Determine the value of $D_2^3(G)$ for complete multipartite graphs $G$.
\end{problem}

\section{Acknowledgments}

The authors would like to thank Cassie Murley for valuable discussions and contributions to this problem. Furthermore, the authors would like to thank the Graduate Research Workshop for Combinatorics (GRWC2019) for facilitating the collaboration of the authors.

\end{document}